\documentclass[12pt]{article}
\usepackage[left=3cm,right=3cm,top=3cm,bottom=3cm]{geometry}
\usepackage{amsmath,amsthm,amssymb} 
\usepackage{graphicx}
\usepackage{float}
\usepackage{authblk}
\usepackage{scalerel}
\RequirePackage{filecontents}
\usepackage{url}

\graphicspath{}
\newtheorem{theorem}{Theorem}
\newtheorem{lemma}{Lemma}

\begin{document}
\author{Dwight Nwaigwe}
\title{On the Convergence of WKB Approximations of the Damped Mathieu Equation}
\date{}
\maketitle

\section*{\centering {Abstract}}
Consider the differential equation  ${ m\ddot{x} +\gamma \dot{x} -x\epsilon \cos(\omega t) =0}$, $0 \leq t \leq T$. One application of it is the modelling of ions in quadrupole traps. The form of the fundamental set of solutions, $P_1(t)e^{\lambda_1 t} ,P_2(t)e^{\lambda_2 t} $ are determined by Floquet theory. In the limit as $m \to 0$ we can apply WKB theory to get first order approximations of this fundamental set: $x_{1,WKB}(t),x_{2,WKB}(t)$.  WKB theory states that this approximation gets better as $m \to 0$ in the sense that $\| P_i(t)e^{\lambda_i t}-x_{i,WKB}(t)\|_{\infty}$ is bounded as function of $m$ for a given $T$, where $x_{i,WKB}(t)$ is a fundamental solution obtained from WKB analysis. However, convergence of the component functions $e^{\lambda_i}$ and $P_i(t)$ are not addressed. We show that $\lambda_i$ and $P_i(t)$ converge to that predicted by WKB theory. We also provide a rate of convergence that is not dependent on $T$.

\section*{\centering {Introduction}}
Consider the damped Mathieu's equation,

\begin{equation}
\label{dampedmathieu}
{ m\ddot{x} +\gamma \dot{x} -x\epsilon \cos(\omega t) =0}
\end{equation} 

where $m, \gamma >0$.  This equation is especially used to model the dynamics of a particle in a quadrupole trap \cite{me}. In such traps, $m$ represents mass, $\gamma$  damping, and $\epsilon$ represents an effective charge of the particle. Matheiu's equation does not have an analytic solution. For small $m$, one may use asymptotic techniques to approximate it, among them the WKB method. For an introduction to WKB one may consult \cite{bo}.  It has been shown by \cite{olver} that the WKB method produces an approximate fundamental set of solutions.  However, the error bound given by \cite{olver} only pertains to each linearly independent solution as whole, not to the components $P_i(t), \lambda_i$ of the Floquet solution $P_i(t)e^{\lambda_i t}$. It may be important to know how well the WKB solution, if at all, approximates these components.  To illustrate, let us start with \ref{dampedmathieu} and use the Liouville transformation  $x(t)=v(t)e^{-\frac{\gamma t}{2m}}$ to arrive at $ m^2 \ddot{v} - \left(\frac{\gamma^2}{4}  + m \epsilon  \cos(\omega t)\right) v=0$. Now that the damped term is gone, this equation is in the appropriate form to use the WKB method. According to \cite{olver}, for an equation of the form $\ddot{v}=u^2f(t)v$ on the interval $[0, T]$, one may get first order approximations of the two linearly independent solutions $v_1(t), v_2(t)$ by

\begin{align}
\label{olverform}
v_1(t) & = \frac{ C_1}{ f(t)^{1/4} } e^{  u  \int_{0}^{T} { \sqrt{f(s)}ds  } } (1+\epsilon_1) \\
v_2(t) & = \frac{ C_2}{ f(t)^{1/4} } e^{ - u  \int_{0}^{T} { \sqrt{f(s)}ds  } } (1+\epsilon_2) \\
|(\epsilon_i)| & \leq \frac{e^{F_j(u,t)}}{2u}-1 \\
F_1(u,t) & =\int_{0}^{t} \left(\frac{1}{f(t)}\right)^{1/4} \left|\frac{d^2}{dt^2} \left(\frac{1}{f(t)}\right)^{1/4}\right| dt  \\
F_2(u,t) & =\int_{t}^{T} \left(\frac{1}{f(t)}\right)^{1/4} \left|\frac{d^2}{dt^2} \left(\frac{1}{f(t)}\right)^{1/4}\right| dt.
\end{align}

The solution breaks down when $t>>T$ or when $t$ is such that $f(t) \to 0$. For our case, this is not a problem since for sufficiently small $m$, $ \frac{\gamma^2}{4} > m \epsilon \cos(\omega t)$. If we plug in our parameters into \ref{olverform} and then recall the Liouville transformation, we get

\begin{equation}
\begin{split}
x(t)= \frac{ C_1}{ \left(  \frac{\gamma^2}{4} + m \epsilon \cos(\omega t)   \right)^{1/4} } e^{  \frac{1}{m}  \int_{0}^{t} { \sqrt{ \frac{\gamma^2}{4} + m \epsilon \cos(\omega s)}ds - \frac{\gamma t}{2m}  } } \ + \\
\frac{ C_2}{ \left(  \frac{\gamma^2}{4} + m \epsilon \cos(\omega t)  \right )^{1/4} } e^{ - \frac{1}{m}  \int_{0}^{t} { \sqrt{ \frac{\gamma^2}{4} + m \epsilon \cos(\omega s)}ds - \frac{\gamma t}{2m}  } }.
\end{split}
\end{equation}

A Taylor expansion shows  $\sqrt{ \frac{\gamma^2}{4} + m \epsilon \cos(\omega s)} \approx \frac{\gamma}{2} + \frac{m \epsilon \cos(\omega s)}{\gamma} - \frac{m^2 \epsilon^2 {\cos(\omega s)}^2}{\gamma^3} + O(m^3)$. 
From this we get, $\frac{1}{m}  \int_{t_0}^{t} { \sqrt{ \frac{\gamma^2}{4} + m \epsilon \cos(\omega s)}ds  \approx  - \frac{\gamma t}{2m}  }+  \frac{\epsilon sin(\omega t)}{\gamma \omega} -\frac{m \epsilon^2}{2 \gamma^3} \left(t + \frac{\sin(2 \omega t)}{2 \omega}   \right) $. This analysis suggests that for sufficiently small $m$ the characteristic exponents of the Floquet solutions are approximated by $- \frac{\epsilon^2m}{2 \gamma^3}$ and $ - \frac{\gamma}{m}$ while the periodic parts are approximated by $e^{\frac{\sin(\omega t)}{\gamma \omega} }$ and  $e^{-\frac{\sin(\omega t)}{\gamma \omega} }$ respectively. We note that it is possible to have two sequences of functions ${f_m(t)}$ and ${g_m(t)}$ such that $f_m(t) g_m(t)$ converge (in the sense of \cite{olver})  to the  solution of \ref{dampedmathieu} without having $f_m(t)$ and $g_m(t)$ converge to the periodic or exponential parts. Part of the goal of this paper is to rule this out.
In this article, we will show that $\lambda_i$ and $P_i(t)$ can be approximated by that suggested by WKB theory, specifically that the asymptotic error is $O(m^2)$ for $\lambda_i$ and $O(m)$ for $P_i(t)$. We note that these these bounds are independent of $T$.

\section*{\centering {Characteristic Exponents}}

Floquet theory states that the solutions of \ref{dampedmathieu} will contain exponential terms. In this section we find an approximation of the characteristic exponents of \ref{dampedmathieu} as $m \to 0$. Specifically, we show that for small $m$, the characteristic exponents are  $-\frac{m \epsilon^2 }{\gamma^3} +O(m^2)$ and $-\frac{\gamma}{m}+O(m^2)$. The approach used here is the analysis of the infinite Hill determinant. For a quick and thorough exposition one may consult \cite{magwink} \cite{richards}, \cite{ggk}. To avoid confusion, we note that in literature, ``Hill's method" or ``infinite determinate analysis" refers to any technique that involves writing an operator in terms of Fourier series and requiring that the determinant of the resultant matrix is nonzero.  Our context will be clear.

Of the articles that make use of the infinite determinant method in some setting, the author is unaware of any that compute the leading order asymptotics.  Instead, truncation of the infinite matrix is done (to typically 3 by 3) and then an approximation of the determinant is made from this truncation which allows one to get an expression that approximates the characteristic exponent. An example of this can bee seen in \cite{stepan}. The reason for truncation is due to the fact that the determinant is ``infinite", thus requiring one to know what the series (which may be difficult to evaluate) converges to. Depending on the application, truncating the determinant may introduce significant error into calculations involving it. For instance, it will be shown that truncating the determinant introduced here results in the quantity $1-O(m^2)$, whereas considering all terms in the infinite expansion yields $1-O(m)$.  For $m \to 0$, this is a big magnitude difference which can affect computations reliant on the determinant. In \cite{me}, the authors take a logarithm of the determinant and multiply it by an integral. Clearly, if the authors were to use $1-O(m^2)$ instead of $1-O(m)$, this would cause an extra multiplication by $m$ which would drastically change the magnitude and nature of the computed values. 
Consider the Liouville transformation $x(t)=v(t)e^{-\frac{\gamma t}{2m}}$, which transforms \ref{dampedmathieu} into

\begin{equation}
\label{exp1}
	{\ddot{v} - \left( \left(\frac{\gamma}{2m}\right)^2  + \frac{\epsilon}{m} \cos(\omega t) \right)v=0 }.
\end{equation}

Equation \ref{exp1} is of the form 

\begin{equation}
\label{exp2}
	\ddot{v}+ v\sum_{n= -\infty}^{\infty}   \mathbf{ G_n }  e^{in\omega t}=0.
\end{equation}

It is straightforward to see that $\mathbf {G_{0}}=-(\frac{\gamma}{2m})^2, \mathbf {G_{1}} = \mathbf{ G_{-1}}= -\epsilon /2m , \mathbf {G_{i}}=0$ otherwise.  According to Floquet theory, \ref{exp2} has at least one solution of the form $v(t)=e^{\mu t}P(t)$, where $P(t)$ is periodic and of the form $P(t)= \sum_{r= -\infty}^{\infty}   \mathbf{ C_r }  e^{ir  \omega t}$. Substituting for $v(t)$, \ref{exp2} then becomes 

\begin{equation}
\label{exp3}
	(\mu+ir \omega)^2 \mathbf{C_r} + \sum_{n= -\infty}^{\infty}   \mathbf{ G_n } \mathbf{ C_{r-n}}=0
\end{equation}

which is equivalent to

\begin{equation}
\label{eqhillform}
	\mathbf{C_r}+  \sum_{n= -\infty, n\neq 0}^{\infty}   \mathbf{ G_n } \rho_r^{-1}(\mu) \mathbf{ C_{r-n}}=0
\end{equation}
where $\rho_r(\mu) =(\mu + ir \omega)^2 + \mathbf{G_0}$.

The method of infinite Hill determinant states that a nontrivial solution of \ref{eqhillform} exists if the determinant $\Delta(\mu)$ of the infinite system is 0. 

%

\begin{table}[H]
	\[
	\Delta(\mu)=
	\begin{vmatrix}
	& \vdots & \vdots & \vdots & \vdots &\vdots \\[6pt]
	\cdots & 1 & \frac{G_1}{  \rho_{ \scaleto{-2}{4pt} }(\mu)     } & 0 & 0 & 0 & \cdots\\[6pt]
	\cdots & \frac{G_{1}}{\rho_{ \scaleto{-1}{4pt}}(\mu)} & 1 & \frac{G_1}{\rho_{ \scaleto{-1}{4pt}}(\mu)} & 0 &0& \cdots\\[6pt]
	\cdots & 0 &  \frac{G_{1}}{\rho_{ \scaleto{0}{4pt}}(\mu)} & 1 &  \frac{G_1}{\rho_{ \scaleto{0}{4pt}}(\mu)} & 0 & \cdots\\[6pt]
	\cdots & 0 &  0 & \frac{G_{1}}{\rho_{ \scaleto{1}{4pt}}(\mu)} & 1 & \ \ \ \ \frac{G_{1}}{\rho_{ \scaleto{1}{4pt}}(\mu)} \cdots\\[6pt]
	\cdots & 0 &  0 & 0 &   \frac{G_{-1}}{\rho_{ \scaleto{2}{4pt}}(\mu)} & 1 & \cdots\\[6pt]
	& \vdots & \vdots & \vdots & \vdots & \vdots
	\end{vmatrix}=0, \rho_r \neq 0
	\]
	\caption{Coefficient matrix of equation \ref{eqhillform}.}
\end{table}

Now using the complex analysis argument referred to in \cite{ggk},\cite{magwink}, or \cite{richards} existence of a nontrivial solution implies that

\begin{equation}
	\cos(i 2 \pi \mu / \omega)=1-2 \Delta(0) sin^2{\pi \sqrt{ \mathbf{G_0}} / \omega }.
\end{equation}

\begin{figure}
\end{figure}

Since $\mathbf {G_{i}}$ and $\rho_{r}(0)$ are real, so is  $\Delta(0)$. Let $\mu=c + id$. Then inserting this into the previous equation and separating real and imaginary parts gives

\begin{equation}
	\frac{e^{2\pi c / \omega}  + e^{-2\pi c / \omega}} {2} \cos(2 \pi d / \omega)= \Re \{  1-2\Delta(0)\sin(\pi \sqrt\mathbf {G_{0}} / \omega)^2 \} 
\end{equation}

\begin{equation}
	\frac{e^{2\pi c / \omega}  - e^{-2\pi c / \omega}}{2} \sin(2 \pi d / \omega)= \Im \{ 1-2\Delta(0)\sin(\pi \sqrt{\mathbf {G_{0}}} / \omega)^2 \}.
\end{equation}

\vskip .6cm
But since $\Delta(0)$ and  $\sin(\pi \sqrt{\mathbf {G_{0}}} / \omega)^2=(-e^{- \pi \gamma / \omega m} + 2 - e^{ \pi \gamma / \omega m})/4$ are real,  ${\Im \{ 1-2\sin(\pi \sqrt{\mathbf {G_{0}}} / \omega)^2 \}}$ is 0, and it follows that $\sin(2 \pi d / \omega)=0$, implying  $d=n \omega /2$ where $n$ is an integer.  We note that it is not possible that $c=0$ since equation 13 will have no solution for small $m$. In addition, if $c=0$, there is no consistency with the WKB method. Thus, $\cos(2 \pi d / \omega)= \pm 1$, and equation 13 becomes

\begin{equation}
	\pm \frac{e^{2\pi c / \omega}  + e^{-2\pi c / \omega}} {2}  =  1-2\Delta(0) \left(\frac{-e^{- \pi \gamma / \omega m} + 2 - e^{ \pi \gamma / \omega m} }{4} \right).
\end{equation}

Clearly, a negative sign in the left hand side of equation 15 makes the equation unsolvable when $m$ is sufficiently small, so we take a positive sign. We rewrite equation 15 as

\begin{equation}
	\cosh(2\pi c / \omega)=1-2\Delta(0)+\Delta(0)\cosh(\pi \gamma / \omega m)
\end{equation}

Noting that for small $m$ we have $\Delta(0) \approx 1$, and that $\cosh(\pi \gamma / \omega m)$ is large, we may write 
\begin{equation}
	c=\pm \frac{\omega}{2\pi}\cosh^{-1}(\Delta(0)\cosh(\pi \gamma / \omega m) ) +err.
\end{equation}

where $err$ denotes the resultant exponentially decaying error. From now on, $err$ refers to any exponentially decaying error. Using $\cosh^{-1}(x)=\log(x+\sqrt{x^2-1})$, a Taylor expansion results in

\begin{equation}
	c=\pm \frac{\omega}{2\pi} \{\log(2\Delta(0)\cosh(\pi \gamma / \omega m)-\frac{1}{2\Delta(0)\cosh(\pi \gamma / \omega m)} +err)\}.
\end{equation}

We further re-write and use another Taylor expansion to get
\begin{equation}
	\begin{split}
		c=\pm \frac{\omega}{2\pi}\{\log(\Delta(0)) + \log(e^ {\pi \gamma / \omega m}  +e^{-\pi \gamma / \omega m})   +err\} &= \\ \pm \frac{\omega}{2\pi}\{\log(\Delta(0)) + \pi \gamma / \omega m + err\} &=  \\ \pm \frac{\omega}{2\pi} \log(\Delta(0)) + \gamma /2m + err.
	\end{split}
\end{equation}

Keeping in mind the Liouville transformation, it follows as $m \to 0$, the larger and smaller characteristic exponents, are respectively,
\begin{align}
	\label{expmax}
	\lambda_{max} & = \Re \{ \mu  -\frac{\gamma  }{2m}  \}_{max} =
	\{ c_{max}  -\frac{\gamma  }{2m}  \} =  \frac{\omega}{2\pi}\{\log(\Delta(0)) +err \}    
	\\
	\label{expmin}
	\lambda_{min} & = \Re \{ \mu  -\frac{\gamma  }{2m}  \}_{min} =
	\{ c_{min}  -\frac{\gamma  }{2m}  \} =- \frac{\gamma}{m}- \frac{\omega}{2\pi}\{\log(\Delta(0)) +err \}.    
\end{align}

\begin{lemma}
$\Delta(0)=1-\frac{m \pi \epsilon^2}{\gamma^3 \omega}+O(m^2)$
\end{lemma}

\begin{proof}
Let $M_{2n+1}$ be the centered, truncated $2n+1$  by $2n+1$ matrix of $\Delta(0)$. Then by induction (via row reduction) one may calculate

\begin{equation}
\label{eqdet}
\det(M_{2n+1})=(f_{2n-1}- f_{2n-3}c_n c_{n-1}) \prod_{i=1}^{n-1}(1-c_{i}c_{i+1}), \ \ n > 3
\end{equation}

\begin{equation}
\label{eqf}
f_{2n-1}:= \frac{\det(M_{2n-1}) }{\prod_{i=1}^{n-1}(1-c_{i}c_{i+1})} \ \ n > 3
\end{equation}

\begin{align}
\det(M_{3})=1-2C_0C_1 \\
\det(M_{5})=(1-C_1C_2)[(1-2C_0C_1)-C_1C_2(1-C_0C_1) ].
\end{align}

From \ref{eqf} it can be seen that $\Delta(0)$ contains the term $1-2\sum_{n=0}^{\infty} c_n c_{n+1}, \ \text{where}  \ c_n=\frac{G_1}{\rho_{n }(0) }$. Let  
	
	\begin{equation}
	S=2\sum_{n=0}^{\infty} c_n c_{n+1}, \ c_n=\frac{G_1}{  \rho_{n }(0)     }=  \frac{\epsilon/2m}{(n\omega)^2+(\gamma/2m)^2}.
	\end{equation}
	
	We have 
	
	\begin{align}
	S & < 2\sum_{n=0}^{\infty} c^2_n = \sum_{n=0}^{\infty} \frac{\epsilon^2/2}{m^2 n^4 \omega^4 +   n^2 \omega^2 \gamma^2/2 +\gamma^4/m^4 } \\
	& =\frac{\epsilon^2}{2m^2 \omega^4} \sum_{n=0}^{\infty} \frac{1}{n^4 +   \frac{n^2 \gamma^2}{2 m^2 \omega^2} + \frac{\gamma^4}{16 m^4 \omega^4}} \\
	&= \frac{\epsilon^2}{2m^2 \omega^4} \sum_{n=0}^{\infty} \frac{1}{\left(n^2+ \frac{\gamma^2}{4 m^2 \omega^2} \right)^2}.
	\end{align}
	
	It is known \cite{listfuncs} that a series of the form $\sum_{n=0}^{\infty} \frac{1}{\left(n^2+ a \right)^2 +z^2}$ is equal to

\begin{equation}
\frac{i \pi}{4z} \left( \frac{\coth \pi \sqrt{a+iz}}{\sqrt{a+iz}} -  \frac{\coth \pi \sqrt{a-iz}}{\sqrt{a-iz}}  \right) -\frac{1}{2(a^2+z^2)}.
\end{equation}

In light of this, we evaluate the above expression for $a=\frac{\gamma^2}{4m^2 \omega^2}$ as $z \to 0$. The $\coth$ term results $\frac{2 \pi m^3 \omega^3}{\gamma^3}$ while the other term gives $O(m^4)$. Combining and multiplying by $\frac{\epsilon^2}{2m^2 \omega^4}$ we get

\begin{equation}
S < \frac{m \pi \epsilon^2}{\gamma^3 \omega}+O(m^2).
\end{equation}

We now bound $S$ from below. Write

\begin{align}
S & > 2\sum_{n=0}^{\infty} c_{n+1}^2 \\
&= \sum_{n=0}^{\infty} \frac{\epsilon^2/2}{m^2 (n+1)^4 \omega^4 +   (n+1)^2 \omega^2 \gamma^2/2 +\gamma^4/m^4 } \\
& =\frac{\epsilon^2}{2m^2 \omega^4} \sum_{n=0}^{\infty} \frac{1}{(n+1)^4 +   \frac{n^2 \gamma^2}{2 m^2 \omega^2} + \frac{\gamma^4}{16 m^4 \omega^4}} \\
&= \frac{\epsilon^2}{2m^2 \omega^4} \sum_{n=0}^{\infty} \frac{1}{\left((n+1)^2+ \frac{\gamma^2}{4 m^2 \omega^2} \right)^2} \\
&= \frac{\epsilon^2}{2m^2 \omega^4}  \left( \sum_{n=1}^{\infty} \frac{1}{\left(n^2+ \frac{\gamma^2}{4 m^2 \omega^2} \right)^2} + O(m^4) \right).
\end{align}

Using the known expression for the series as stated earlier, we get

\begin{equation}
S > \frac{m \pi \epsilon^2}{\gamma^3 \omega}+O(m^2).
\end{equation}

By the squeeze theorem, we conclude that as $m \to 0, S \to \frac{m \pi \epsilon^2}{\gamma^3 \omega}+O(m^2)$. Note that $\det M_{2n+1}$ has a leading order $m^2$ term but that $\lim_{n \to \infty} \det M_{2n+1}=\Delta(0)$ has a leading order $m$ term. This illustrates how matrix truncation affects the solution to \ref{dampedmathieu}. 

Since we now know that the infinite sum of $O(m^2)$ terms converge to as a function of $m$, we now study what the infinite sums of $O(m^4)$ terms converge to. Consider the expression $\prod_{i=1}^{n}(1-c_{i-1}c_i)$ on the right hand side of \ref{eqdet}. The sum of $O(m^{2p})$ terms that come from it are less than $\left(\sum_{i=0}^n c_ic_{i+1} \right)^p<\left( \frac{S}{2}\right)^p=O(m^p)$. A similar analysis holds for $(f_{2n-1}- f_{2n-3}c_n c_{n-1})$. Letting $p=2$, we find that the infinite sum of $O(m^4)$ terms for any $M_{2n+1}$ is $O(m^2)$. As a result we have 

\end{proof}

\begin{theorem}
	The characteristic exponents of \ref{dampedmathieu} are asymptotically equal to that suggested by a first order WKB approximation, i.e, $\lambda_{max}= -\frac{m\epsilon^2}{ 2\gamma^3}+O(m^2)$ and $\lambda_{min}= -\frac{\gamma}{m}+O(m^2)$.
\end{theorem}

\begin{proof}
Taylor expand $log(\Delta(0))$ and use \ref{expmin} and \ref{expmax}.
\end{proof}

\section*{\centering {Periodic Parts}}

Set $\lambda=\lambda_{max}$. Since there exists a solution of \ref{dampedmathieu} in the form $P(t)e^{\lambda t}$, we plug this into \ref{dampedmathieu} and obtain


\begin{equation}
\label{period2}
{ m\ddot{P} +\gamma' \dot{P} + P\epsilon'  =0}.
\end{equation}

where $\gamma'=\gamma + 2m\lambda+, \ \epsilon'=  -\epsilon \cos(\omega t) + \gamma \lambda + m \lambda^2. $

We get two solutions of \ref{period2}, denoted (via an abuse of notation since the function $P_2(t)$ is not periodic as suggested. )  as $P_1(t)=P_{max}(t)$  and $P_2(t)=P_{min}(t)e^{(\lambda_{min} -\lambda_{max})t }$, where $P_{max}(t)$ is the periodic part that goes with $ e^{\lambda_{max}t}$ and similarly for $P_{min}(t)$ and $ e^{\lambda_{min}t}$.  By \cite{olver}, we can find an error bound on the solutions $P_1(t),  P_2(t)$ of \ref{period2} for the given parameters. For $\lambda_{max}(m) \in [a,b]$, we can obtain a supremum of these errors as a function of $m$. Consequently, our problem of bounding  $P_{max}(t)$ has been reduced to bounding $\lambda_{max}$. We apply this fact in the next theorem.

\begin{theorem}
	Denote $P_{max}(t)$ to be the periodic function that goes with $ e^{\lambda_{max}t}$ and similarly denote $P_{min}(t)$. Then $\lVert P_{max}(t)- \frac{ 1}{ \left( \frac{\gamma^2}{4}  + m \epsilon  \cos(\omega t) \right) ^{1/4} }e^{\frac{\sin(\omega t)}{\gamma \omega} }\rVert_{\infty}=O(\frac{m}{\omega})$ and \\ ${\lVert P_{min}(t)- \frac{ 1}{ \left( \frac{\gamma^2}{4}  + m \epsilon  \cos(\omega t) \right) ^{1/4} }e^{-\frac{\sin(\omega t)}{\gamma \omega} }\rVert_{\infty}=O(\frac{m}{\omega})}$. 
\end{theorem}

\begin{proof}

We show this as follows. Use the transformation $P(t)=v(t)e^{-\frac{\gamma' t}{2m}}$ on \ref{period2}: 

\begin{equation}
  m^2 \ddot{v} -\left( \frac{\gamma'^2}{4}  + m \epsilon'  \right) v=0.
\end{equation}
 
 From earlier, the asymptotic solutions to the above equation are
 
 \begin{equation}
 v_i(t)=\tilde{v_i(t)}+\tilde{v_i(t)}\delta_i(t),\  i=1,2
 \end{equation}
 
 where
 
 \begin{align}
\tilde{v_1(t)} &= \frac{ 1}{ \left( \frac{\gamma'^2}{4}  + m \epsilon'   \right) ^{1/4} } e^{  \frac{1}{m}  \int_{t_0}^{t} { \sqrt{\left( \frac{\gamma'^2}{4}  + m \epsilon'  \right) }ds  } } \\
\tilde{v_2(t)} &=
 \frac{ 1}{ \left( \frac{\gamma'^2}{4}  + m \epsilon'   \right) ^{1/4} } e^{ -  \frac{1}{m}  \int_{t_0}^{t} { \sqrt{\left( \frac{\gamma'^2}{4}  + m \epsilon'   \right) }ds  } } .
 \end{align}

\begin{equation}
|\delta_i(t)| \lessapprox e^{\frac{5 \epsilon \omega m^2 t}{2 \gamma^3}}-1.
\end{equation}

Keeping in mind the Liouville transformation, the two linearly independent solutions of \ref{period2} are thus

\begin{align}
P_1(t) &=\frac{ 1}{ \left( \frac{\gamma'^2}{4}  + m \epsilon'  \right) ^{1/4} } e^{  \frac{1}{m}  \int_{t_0}^{t} { \sqrt{\left( \frac{\gamma'^2}{4}  + m \epsilon'   \right) }ds  } -\frac{\gamma' t}{2m} } \left( 1+ \delta_1(t)\right) \\
P_2(t) &=\frac{ 1}{ \left( \frac{\gamma'^2}{4}  + m \epsilon'   \right) ^{1/4} } e^{-  \frac{1}{m}  \int_{t_0}^{t} { \sqrt{\left( \frac{\gamma'^2}{4}  + m \epsilon'   \right) }ds  } -\frac{\gamma' t}{2m} } \left( 1+ \delta_2(t)\right).
\end{align}

Let $\lambda= \lambda_{max}$.  Then it is clear that one of $P_i(t)$ is equal to $P_{max}(t)$, and the other $P_i(t)$ is equal to $P_{min}(t)e^{(\lambda_{min} -\lambda_{max})t } $. By inspection, $P_1(t)=P_{max}(t) $.  As stated before, the WKB method suggests $P_{max} \approx \left(\frac{\gamma^2}{4}\right)^{-\frac{1}{4}} e^{\frac{\sin(\omega t)}{\gamma \omega} }$. We verify the accuracy of this approximation by bounding 

\begin{equation}
\label{periodb1}
\lVert P_{1}(t)-   \left(\frac{\gamma^2}{4} \right) ^{-1/4}  e^{\frac{\sin(\omega t)}{\gamma \omega} }\rVert_{\infty}.
\end{equation}

Since the periodic functions have period $\frac{2 \pi}{\omega}$, \ref{periodb1} is equivalent to bounding

\begin{equation}
\label{periodb2}
\lVert P_{1}(t)-   \left(\frac{\gamma^2}{4} \right) ^{-1/4}  e^{\frac{\sin(\omega t)}{\gamma \omega} }\rVert_{t \in \left[0 \ \frac{2 \pi}{\omega}\right] }.
\end{equation}

Doing a Taylor expansion of $P_1(t)$, and using the fact that $|\lambda_{max}|=O(m)$, we find that 

\begin{equation}
\lVert P_{1}(t)-  \frac{ 1}{ \left( \frac{\gamma^2}{4}  + m \epsilon  \cos(\omega t) \right) ^{1/4} } e^{\frac{\sin(\omega t)}{\gamma \omega} }\rVert_{t \in \left[0 \ \frac{2 \pi}{\omega} \right] }=O\left(\frac{m}{\omega}\right).
\end{equation}

\end{proof}
\textbf{remark:}
A similar procedure applies to bounding $P_{min}(t)$.

\section*{\centering {\small{Acknowledgements}}}
I would like to thank Janek Wehr for fruitful discussions. This work was partially supported by NSF grant DMS 1615045.

\begin{center}
\small{\bibliography{convergencebib}}
\end{center}	
\bibliographystyle{plain}		
\nocite{*}

\end{document}